\documentclass[a4paper,11pt,english]{article}
\usepackage{amsmath}
\usepackage{amsfonts}
\usepackage{amssymb}
\usepackage{amsthm}
\usepackage[all]{xy}
\usepackage{geometry}
\usepackage{babel}
\newtheorem {thm}{Theorem}
\newtheorem* {thm*}{Theorem}

\newtheorem* {cor*}{Corollary}
\newtheorem {lem}[thm]{Lemma}
\newtheorem {prop}[thm]{Proposition}

\newtheorem {rem}[thm]{Remark}

\theoremstyle{definition}
\newtheorem {exa}[thm]{Example}
\newtheorem* {conj*}{Conjecture}

\DeclareMathOperator{\End}{End}

\DeclareMathOperator{\ord}{ord}

\DeclareMathOperator{\Hom}{Hom}
\DeclareMathOperator{\lcm}{lcm}

\author{Antonella Perucca}
\title{The multilinear support problem for products of abelian varieties and tori}
\date{}

\begin{document}
\maketitle

\begin{abstract}
Let $G$ be the product of an abelian variety and a torus defined over a number field $K$. The aim of this paper is detecting the dependence among some given rational points of $G$ by studying their reductions modulo all primes of $K$. We show that if some simple conditions on the order of the reductions of the points are satisfied then there must be a dependency relation over the ring of $K$-endomorphisms of $G$. We generalize Larsen's result on the support problem to several points on products of abelian varieties and tori.
\end{abstract}

\section{Introduction}

For an algebraic group defined over a number field, it is natural to ask which global properties of the rational points can be detected by studying their reductions. 
In this paper we are concerned with the support problem, which asks to what extent dependency relations (over the endomorphism ring) can be derived from conditions on the order of the reductions of the points. For a detailed history of the support problem and the known results, see \cite{Peruccasupp1} and \cite{DemeyerPerucca}.

The original support problem can be stated as follows: if two rational points $P$ and $Q$ are given, is it true that $Q$ is the image of $P$ via an endomorphism provided that, for all but finitely many reductions, the order of the reduction of $Q$ divides the order of the reduction of $P$?
The answer is in general negative for abelian varieties and for tori, as was shown respectively by Larsen in \cite{Larsen03} and by the author and Demeyer in \cite{DemeyerPerucca}.
Nevertheless, for products of abelian varieties and tori we proved in \cite{Peruccasupp1} that there exist a non-zero integer $c$ and an endomorphism $\phi$ such that $\phi (P)=cQ$ (for abelian varieties this result is due to Larsen, see \cite{Larsen03}).

This last property is false for not necessarily split semi-abelian varieties, as shown to the author by Peter Jossen: the reason is that there are extensions of abelian varieties by split tori of dimension $>1$ which have endomorphism ring $\mathbb Z$.
Thus it is justified that in what follows we restrict our attention to split semi-abelian varieties. 

We study the generalization of the support problem to several points, namely the {\em multilinear support problem}. This variant was first considered by Bara\'nczuk in \cite{Baranczuk06} and \cite{Baranczuk08} and by the author in \cite[Section 5]{Peruccasupp1}. We are able to generalize Larsen's result \cite[Theorem 1]{Larsen03}, thus solving the multilinear support problem for products of abelian varieties and tori:

\begin{thm}\label{bellothm}
Let $G$ be the product of an abelian variety and a torus defined over a number field $K$. Let $P_1,\ldots, P_n, Q_1, \ldots, Q_m$ be points in $G(K)$. Suppose that, for a set of primes $\mathfrak p$ of $K$ of Dirichlet density $1$, the following holds: there exist indexes $i$ and $j$ (depending on the prime $\mathfrak p$) such that 
$$\ord (Q_i \bmod \mathfrak p)\mid  \ord (P_j \bmod \mathfrak p)\,.$$
Then there exist indexes $i$ and $j$ such that $cQ_i=\phi(P_{j})$ for some $K$-endomorphism $\phi$ of $G$ and some non-zero integer $c$.
\end{thm}

It is possible to consider points on different algebraic groups, as shown in Section~\ref{proofs}. In particular we obtain the existence of $K$-homomorphisms between the considered algebraic groups.

By Theorem~\ref{plus}, we can weaken the condition of Theorem~\ref{bellothm} as follows: there exists an index $j$ (depending on $\mathfrak p$) such that
$$\gcd \{\ord (Q_i \bmod \mathfrak p)\} \mid \ord (P_j \bmod \mathfrak p)\,.$$
Example~\ref{lcmorder} shows that one cannot similarly weaken the condition of Theorem~\ref{bellothm} as
$$ \ord (Q_j \bmod \mathfrak p)\mid  \lcm \{\ord (P_i \bmod \mathfrak p)\}\,.$$

Nevertheless we prove the following statement, where $\ord_{\ell}$ means the $\ell$-adic valuation of the order:

\begin{thm}\label{minmax}
Let $G$ be the product of an abelian variety and a torus defined over a number field $K$. 
Let $P_1,\ldots, P_{n},Q_1,\ldots, Q_{m}$ be points in $G(K)$.
Let $\ell$ be a rational prime. Suppose that, for a set of primes $\mathfrak p$ of $K$ of Dirichlet density $1$, the following holds:
$$\min_{i=1,\ldots,m}\; \{\ord_\ell(Q_i \bmod \mathfrak p)\}\,\leq \max_{j=1,\ldots,n}\; \{\ord_\ell(P_j \bmod \mathfrak p)\}\,.$$
Then there exists an index $i$ such that a non-zero multiple of $Q_i$ belongs to the $\End_K G$-submodule of $G(K)$ generated by the points $P_1,\ldots, P_n$.
\end{thm}

This theorem strengthens in several ways the first assertion of \cite[Theorem 1.1]{Baranczuk08} by Bara\'nczuk.

The results in this paper are related to calculating the rank of the Mordell-Weil group of abelian varieties. Indeed, the proofs are based on Chebotarev density theorem so by using effective versions of this theorem it would be possible to check the conditions in the statements only for finitely many reductions (under GRH, see \cite[Proposition 4.1]{Murty}). 
Since we apply Chebotarev's theorem to Kummer extensions, the ramification is controled (\cite[Lemma 5]{Peruccaord2}). However, to estimate the degree of these number fields one should first obtain explicit bounds in results on Kummer theory and the $\ell$-adic representation of abelian varieties (\cite[Theorem 1]{Bertrand} and \cite[Corollaire 1]{Bogomolov}, which are used in the proof of \cite[Theorem 7]{Peruccaord1}).

\section{Preliminaries}

In this section, $K$ will denote a given number field. Let $G$ be the product of an abelian variety and a torus defined over $K$. We say that a point in $G(K)$ is {\em independent} if it generates a free $\End_K G$-submodule of $G(K)$ of rank $1$. 
Let $G_1,\ldots, G_n$ be products of abelian varieties and tori defined over $K$. For every $i=1,\ldots, n$ let $R_i$ be a point in $G_i(K)$. The points $R_1,\ldots, R_n$ are said to be {\em independent} if the point $R=(R_1,\ldots, R_n)$ in $\prod G_i(K)$ is independent. We list some basic properties:

\begin{lem}\label{lemma}
Let $R$ be a point in $G(K)$. \begin{enumerate}

\item The point $R$ is independent if and only if it has a non-zero multiple which is independent. 

\item The point $R$ is independent if and only if the Zariski closure of $\mathbb Z R$ in $G$ is $G$ itself or, equivalently, if the smallest $K$-algebraic subgroup of $G$ containing $R$ is $G$.

\item Suppose that $G$ is $K$-simple (i.e. without proper $K$-algebraic subgroups of positive dimension). If $R$ has infinite order then it is independent.

\item The point $R$ generates a free $\End_K G$-submodule of $G(K)$ of rank $1$ if and only if it generates a free $\End_{\bar{K}} G$-submodule of $G(\bar{K})$ of rank $1$.

\item Let $R_1,\ldots, R_n$ be points in $G(K)$. The points $R_1,\ldots, R_n$ are independent if and only if they generate a free $\End_K G$-submodule of $G(K)$ of rank $n$. 

\item Let $n$ be a positive integer. Up to replacing $K$ by a finite extension, the group $G(K)$ contains at least $n$ independent points.

\item Let $G'$ be the product of an abelian variety and a torus defined over $K$. Let $\alpha$ be a $K$-homomorphism from $G$ to $G'$. If $\alpha(G)=G'$ then $R$ in $G(K)$ is independent only if $\alpha(R)$ in $G'(K)$ is independent. If $\dim G\leq \dim G'$ then $R$ in $G(K)$ is independent if $\alpha(R)$ in $G'(K)$ is independent. In particular, if $\alpha$ is a $K$-isogeny then $R$ in $G(K)$ is independent if and only if $\alpha(R)$ in $G'(K)$ is independent. \end{enumerate}
\end{lem}

\begin{proof} For (2), see \cite[Section 2]{Peruccaord1}; (3) is a consequence of (2); (4) follows from \cite[Proposition 1.5]{Ribet}; (1) and (5) are immediate from the definition of independence; (6) comes from the fact that $G(\bar{K})$ has infinite rank and $\End_{\bar{K}} G$ is a finitely generated $\mathbb Z$-module. 

We prove (7) by using (2). For the first assertion, if the smallest $K$-algebraic subgroup of $G'$ containing $\alpha(R)$ is not $G'$ then the preimage by $\alpha$ of this group is not $G$ and it contains $R$. For the second assertion, if $\alpha(R)$ is independent in $G'$ then the smallest $K$-algebraic subgroup of $G$ containing $R$ has image $G'$ via $\alpha$ so we conclude by comparing its dimension with that of $G$.
\end{proof}

The following Lemma is a refinement of \cite[Lemma 5]{Peruccasupp1}:

\begin{lem}\label{Xa}
Let $I=\{1,\ldots, n\}$. For every $i\in I$ let $B_i$ be either a $K$-simple torus or a $K$-simple abelian variety. Suppose that for $i\neq j$ either $B_i=B_j$ or $\Hom_K(B_i,B_j)=\{0\}$. 
For every $i\in I$ let $R_i$ be a point in $B_i(K)$. Unless all the points $R_i$'s are torsion, there exists a non-empty subset $J$ of $I$ satisfying the following properties: the points $\{R_i\}_{i\in J}$ are independent and if $J'\subseteq I$ contains properly $J$ then the points $\{R_i\}_{i\in J'}$ are not independent; for every $i$ in $I$ we can write
$$cR_i=\sum_{j\in J}\alpha_{ij} R_j$$
where $c$ is a non-zero integer, $\alpha_{ij}$ in $\Hom_K (B_j,B_i)$ and $\alpha_{ij}=0$ whenever $j>i$.
\end{lem}
\begin{proof} We do induction on $n$. If $n=1$ then the point $R_1$ has infinite order hence by Lemma~\ref{lemma} we have $J=I=\{1\}$. Suppose that we know the result for $n$ and write $I'=I\cup \{n+1\}$ and $J'=J\cup \{n+1\}$. If the points $R_j$ for $j\in J'$ are independent then $J'$ is the maximal subset we are looking for and we can obviously express $c R_{n+1}$ as requested. Now consider the other case, namely that the points $R_j$ for $j\in J'$ are not independent. Since the points $R_j$ for $j\in J$ are independent, there must be a relation of the form 
$$\alpha R_{n+1}=\sum_{j\in J}\alpha_{n+1, j} R_j$$
where $\alpha$ is in $\End_K (B_{n+1})$ and it is non-zero, and where $\alpha_{n+1,j}$ is in $\Hom_K (B_{j}, B_{n+1})$ (see \cite[Lemma 4]{Peruccasupp1}). Since $\alpha$ is an isogeny, let $d$ be its degree. We can replace $c$ so that it is a multiple of $d$. Thus up to composing with an element of $\End_K (B_{n+1})$ we may assume that $\alpha=[c]$ and we can write
$$cR_{n+1}=\sum_{j\in J}\alpha_{n+1,j} R_j\,.$$
This proves the inductive step.
\end{proof}
 
The following Lemma is a refinement of \cite[Corollary 14]{Peruccaord1}:

\begin{lem}\label{Xb}
We keep the notations of Lemma~\ref{Xa}. Suppose that all the points $R_i$'s have infinite order. Let $\ell$ be a prime number and let $m$ be a positive integer. Then, for every $j \in J$, there exists an integer $m_{j}$ such that the following relation holds for the $\ell$-adic valuation of the order:
$$\ord_{\ell}(R_j \bmod \mathfrak p)=m_j\quad\forall j\in J\quad \Rightarrow\quad  \ord_{\ell}(R_i \bmod \mathfrak p)>m\quad\forall i\in I\quad$$
for all but finitely many primes $\mathfrak p$ of $K$.
\end{lem}
\begin{proof} Let $\alpha_{ij}$'s and $c$ be as in Lemma~\ref{Xa}. The maps $\alpha_{ij}$ which are non-zero are isogenies. Call $d$ the maximum of the $\ell$-adic valuation of their degrees.
Without loss of generality suppose that $J=\{1,\ldots, r\}$. Take $m_1$ such that $(m_1-m)>d$ and for every $j=2,\ldots, r$ take $m_j$ such that $(m_j-m_{j-1})>d$.

Let $\mathfrak p$ be a prime of $K$ for which the reductions of the given points are well-defined and suppose that $ord_\ell(R_j \bmod \mathfrak p)=m_j$ for every $j\in J$. 
Up to excluding finitely many primes $\mathfrak p$, \cite[Lemma 3]{Peruccasupp1} implies the following:
$$ord_\ell(R_j \bmod \mathfrak p)-d\leq ord_\ell( \alpha_{ij} R_{j} \bmod \mathfrak p)\leq ord_\ell(R_j \bmod \mathfrak p)$$ whenever $\alpha_{ij}\neq 0$ (in this case $\alpha_{ij}$ is an isogeny).
Consequently, for any $i>r$ we have $$ord_\ell(cR_i \bmod \mathfrak p)=ord_\ell( \alpha_{ij} R_{j} \bmod \mathfrak p)$$ where $j$ is the greatest index such that $\alpha_{ij}\neq 0$ (since $R_i$ has infinite order, not all $\alpha_{ij}$'s are zero). We deduce that $ord_\ell(R_i \bmod \mathfrak p)$ is greater than $m$.
\end{proof}

The following Lemma is a generalization of \cite[Proposition 12]{Peruccaord1}. The improvement is that, when working with different rational primes, we are allowed to change the set of points we are considering.

\begin{lem}\label{XXX}
Let $S$ be a finite set of rational primes. For every $\ell\in S$ let $\{R_j\}_{j\in J_\ell}$ be a set of independent points, such that each point lies in some product of an abelian variety and a torus defined over $K$.
For every $\ell\in S$ and for every $j\in J_{\ell}$ let $a_{j}$ be a non-negative integer. Then there exists a positive density of primes $\mathfrak p$ of $K$ such that the following holds for every $\ell\in S$:
$$\ord_{\ell}(R_j \bmod \mathfrak p)=a_{j}\quad \text{for every $j\in J_\ell$}\,.$$
\end{lem}
\begin{proof}
Let $R_j$ belong to $G_{j}(K)$. For every $\ell\in S$ and for every $j\in J_\ell$, let $T_{j}$ be a point of order $\ell^{a_{j}}$ which lies in $G_{j}(\bar{K})$. Let $F$ be a finite Galois extension of $K$ where all these torsion points are defined. 
We are interested in finding primes $\mathfrak p$ of $K$ such that for some prime $\mathfrak q$ of $F$ over $\mathfrak p$ we have for every $\ell\in S$:
$$\ord_{\ell}(R_j-T_{j} \bmod \mathfrak q)=0\quad \text{for every $j\in J_\ell$}\,.$$
To study this property, we can clearly consider the corresponding point $R_{\ell}$ in $\prod_{j\in J_{\ell}} G_{j}(F)$ and require that 
$$\ord_{\ell}(R_{\ell} \bmod \mathfrak q)=0\,.$$
Since the point $R_\ell$ is independent for every $\ell\in S$, we conclude by applying Lemma~\ref{newthm7}.
\end{proof}

\begin{lem}\label{newthm7}
Let $S$ be a finite set of rational primes. For every $\ell\in S$, let $G_\ell$ be the product of an abelian variety and a torus defined over $K$. Let $F$ be a finite extension of $K$. For every $\ell\in S$ let $R_\ell$ be an $F$-rational point on $G_\ell$ such that the Zariski closure of $R_{\ell}$ in $G_{\ell}$ is connected.
There exists a positive Dirichlet density of primes $\mathfrak p$ of $K$ such that the following holds: for some prime $\mathfrak q$ of $F$ over $\mathfrak p$ the order of $(R_{\ell} \bmod \mathfrak q)$ is coprime to $\ell$ for every $\ell\in S$.
\end{lem}
\begin{proof}
We claim that it is straight-forward to adapt the proof of \cite[Theorem 7]{Peruccaord1}: indeed, the Kummer fields that we consider for the different points $R_{\ell}$ have degree a power of $\ell$ and so they are linearly disjoint.
\end{proof}

\section{Proof of the results}\label{proofs}

In this section, we prove the results mentioned in the Introduction. The following statement implies Theorem~\ref{bellothm}.

\begin{thm}\label{plus}
Let $G$ be the product of an abelian variety and a torus defined over a number field $K$. 
Let $P_1,\ldots, P_{n},Q_1,\ldots, Q_{m}$ be points in $G(K)$. Let $S$ be a set of rational primes with cardinality $n$. Suppose that, for a set of primes $\mathfrak p$ of $K$ of Dirichlet density $1$, there exists an index $j$ (depending on $\mathfrak p$) such that the following holds: for every $\ell\in S$ we have
$$\min_{i=1,\ldots, m} \{\ord_{\ell}(Q_i \bmod \mathfrak p)\} \leq \ord_{\ell}(P_j \bmod \mathfrak p)\,.$$
Then there exist indexes $i$ and $j$ such that $cQ_i=\phi(P_{j})$ for some $\phi\in \End_K G$ and for some non-zero integer $c$.
\end{thm}

We can reformulate Theorems~\ref{plus}~and~\ref{minmax} for points not necessarily on the same algebraic group. The only difference would be that we would not get $K$-endomorphisms, but $K$-homomorphisms between the algebraic groups. To reduce to the known case simply consider the product of the given algebraic groups and replace each point by its image with respect to the inclusion in the product.

\begin{rem}
Let $a$ be a positive integer. In Theorem~\ref{minmax} it suffices to require
$$\min_{i=1,\ldots,m}\; \{\ord_\ell(Q_i \bmod \mathfrak p)\}\,\leq \max_{i=1,\ldots,n}\; \{\ord_\ell(P_i \bmod \mathfrak p)\}+ a\,.$$
\end{rem}
\begin{proof}
We can multiply the points $Q_i$'s by $\ell^{a}$ to reduce to the case $a=0$.
\end{proof}

Notice that one could state a similar remark for  Theorem~\ref{plus}. The following Example shows a problem which arises by weakening the hypotheses of Theorem~\ref{bellothm}.

\begin{exa}\label{lcmorder}
Let $G$ be the product of an abelian variety and a torus defined over a number field $K$. Suppose that $G(K)$ contains points $P_1,P_2$ which are independent, and define $Q=P_1+P_2$. 
The following condition is then satisfied for all but finitely many primes $\mathfrak p$ of $K$: the order of $(Q \bmod \mathfrak p)$ divides the least common multiple of the orders of $(P_1 \bmod \mathfrak p)$ and of $(P_2 \bmod \mathfrak p)$.
However, no multiple of $Q$ is the image of $P_1$ or of $P_2$ by a $K$-endomorphism of $G$. By Lemma~\ref{lemma}, up to a finite extension of the field $K$, we can construct such an example for every $G$.
\end{exa}

\textit{Proof of Theorem~\ref{minmax}.}
\noindent \emph{Step 1.}
We may suppose that $n=1$. Indeed, by replacing $G$ by $G^n$ and the point $Q_i$ by $(Q_i,0,\ldots, 0)$ for $i=1,\ldots, m$, it suffices to prove the statement for the points $P; Q_1,\ldots, Q_m$ where $P=(P_1,\ldots, P_n)$.

It suffices to prove Theorem~\ref{minmax} in the case $G=\prod_{h=1}^e B_h$, where for every $h$ the factor $B_h$ is a $K$-simple abelian variety or a $K$-simple torus and for $h'\neq h$ either $B_{h'}=B_h$ or $\Hom_K(B_{h'},B_h)=\{0\}$. Indeed, by the Poincar\'e Reducibility Theorem, $G$ is $K$-isogenous to such a product, so we can reason as in \cite[Lemma 7]{Peruccasupp1}.

We may assume that $P$ has infinite order because otherwise \cite[Corollary 14]{Peruccaord1} and the condition in the statement imply that some of the points $Q_i$'s is torsion so the theorem holds.

\noindent \emph{Step 2.} 
Define points $R_i$ for $i=1,\ldots, (1+m)e$ as follows:
$$(P,Q_1,\ldots,Q_m)=(R_1,\ldots,R_e; R_{e+1},\ldots, R_{2e}; R_{me+1},\ldots, R_{(1+m)e})\,.$$

Thus the point $R_i$ belongs to $B_h(K)$ if $i\equiv h (\bmod\,e)$.
Apply Lemma~\ref{Xa} to the family $\{R_i\}$. Let $J$ and $c$ be as in the Lemma, and write $J'=J\cap \{1,\ldots, e\}$. 
So we can write
$$cR_i=\sum_{J'}\alpha_{ij} R_j + \sum_{J \setminus J'}\alpha_{ij} R_j$$
where the points $R_j$ for $j\in J$ are independent, and the second sum is zero whenever $i\leq e$.
Notice that $J'\neq\emptyset$ because $P$ is not torsion. We may also suppose $J' \neq J$ because otherwise the statement is clear.

\noindent \emph{Step 3.} 
If $b=v_{\ell}(c)+1$ then we have the following: 
$$\ord_{\ell}(R_j \bmod \mathfrak p)=0\quad \forall j \in J'\; \Rightarrow\; \ord_{\ell}(P \bmod \mathfrak p)<b\,.$$

For every $i>e$, write
$$R_i'=\sum_{J \setminus J'}\alpha_{ij} R_j\,.$$
The points $R_i'$ are either zero or have infinite order because the points $\{R_j\}_{j\in J}$ are independent. We now apply Lemma~\ref{Xb}. Then, for $j\in J\setminus J'$, there exist integers $a_j$ such that the following holds:
$$\ord_{\ell}(R_j \bmod \mathfrak p)=a_j\quad \forall j \in J\setminus J'\; \Rightarrow\; \ord_{\ell}(R_i' \bmod \mathfrak p)>b$$
for every $i$ such that $R_i'$ is non-zero.

Putting things together, if some prime $\mathfrak p$ of $K$ satisfies:
$$\ord_{\ell}(R_j \bmod \mathfrak p)=0\quad \forall j \in J'\quad \text{and}\quad \ord_{\ell}(R_j \bmod \mathfrak p)=a_j\quad \forall j \in J\setminus J'$$
then $$\ord_{\ell}(R_i \bmod \mathfrak p)> b > \ord_{\ell}(P \bmod \mathfrak p)$$
for every point $R_i$ such that $R_i'$ is non-zero.

\noindent \emph{Step 4.} 
Suppose that the statement is false, and so that no point $Q_k$ has a multiple which belongs to the $\End_K G$-submodule of $G(K)$ generated by $P$. Then for every $k=1,\ldots,m$ there exists a point $R_i$ with $ke< i\leq (k+1)e$ such that $R_i'$ is non-zero. In particular, for the primes $\mathfrak p$ considered in the third step we must have 
$$\min_{k=1,\ldots, m} \{\ord_{\ell}(Q_k \bmod \mathfrak p)\}> b > \ord_{\ell}(P \bmod \mathfrak p)\,.$$
By \cite[Proposition 12]{Peruccaord1}, since the points $\{R_j\}_{j\in J}$ are independent, there exists a positive density of primes $\mathfrak p$ of $K$ as requested. Then we contradicted the condition in the statement.
\hfill $\square$
\newline

The following Proposition strengthens assertions contained in \cite[Theorem 1.1]{Baranczuk08}.

\begin{prop}\label{strongbaru=1} We keep the notations of Theorem~\ref{minmax} and suppose that the condition holds for every rational prime $\ell$.
\begin{enumerate}
\item Let $m=1$ (write $Q_1=Q$). If the points $P_1,\ldots,P_n$ are independent then the $\End_K G$-submodule of $G(K)$ that they generate contains $Q$. 
\item Let $n=1$ (write $P_1=P$). If the points $Q_1,\ldots, Q_m$ are independent then exactly one of them is the image of $P$ via a $K$-endomorphism of $G$.
\end{enumerate}
\end{prop}
\begin{proof}
\emph{(1)} The point $P=(P_1,\ldots, P_n)$ in $\prod G_i(K)$ is independent, so it suffices to apply \cite[Corollary 8 and Proposition 9]{Peruccasupp1} to $P$ and $Q$.

\emph{(2)} Without loss of generality, by Theorem~\ref{minmax} we have $cQ_1=\phi(P)$ for some $\phi$ in $\End_K G$ and for some positive integer $c$. Since the points $cQ_1, Q_2,\ldots, Q_n$ are independent then by (7) of Lemma~\ref{lemma} the points $P,Q_2,\ldots, Q_m$ are also independent. Suppose that $c>1$ is minimal and that $\ell$ is a prime factor of $c$. We first show that $[\ell]$ divides $\phi$ in $\End_K G$ by showing the containment of the kernels.  Let $T$ be a non-zero point in $G[\ell]$. To ease notations, suppose that $T\in G(K)$. Consider the positive density of primes $\mathfrak p$ such that $\ord_{\ell}(P-T \bmod \mathfrak p)=0$ and $\ord_{\ell}(Q_i \bmod \mathfrak p)=2$ for every $i>1$ (\cite[Proposition 11]{Peruccaord1}). Taking into account $Q_1$ and since $\ell\mid c$, the condition in the statement forces $\phi(T)=0$. Thus we can write $(c/\ell) Q_1=\phi'(P)+T'$ where $\phi'$ is in $\End_K G$ and $T'$ is in $G[\ell]$. We deduce that $T'=0$ by similarly considering primes $\mathfrak p$ such that $\ord_{\ell}(P \bmod \mathfrak p)=0$ and $\ord_{\ell}(Q_i \bmod \mathfrak p)=1$ for every $i>1$. This contradicts the minimality of $c$. The unicity in the statement is a consequence of the fact that the points $Q_i$'s are independent. 
\end{proof}

Notice that the condition of independence in Proposition~\ref{strongbaru=1} cannot be removed in general (not even if $n=m=1$) because otherwise there would be a contradiction with \cite[Proposition 2]{Larsen03}.\\

\textit{Proof of Theorem~\ref{plus}.}
Analogously to what done for Theorem~\ref{minmax}, we may assume that $G=\prod_{h=1}^e B_h$, where for every $h$ the factor $B_h$ is a $K$-simple abelian variety or a $K$-simple torus and for $h'\neq h$ either $B_{h'}=B_h$ or $\Hom_K(B_{h'},B_h)=\{0\}$.

We may assume that not all points $P_i$'s are torsion because otherwise \cite[Corollary 14]{Peruccaord1} and the condition in the statement imply that some of the points $Q_i$'s is torsion so the theorem holds. We may then suppose that every point $P_i$ has infinite order: without loss of generality, if $P_1$ is a torsion point we can remove it, up to replacing each $Q_i$ by $\ord (P_1)\cdot Q_i$.

For every point $P_i$ choose a different prime $\ell$ of $S$ and write $P_i=P_\ell$. Suppose that no point $Q_i$ has a multiple which belongs to the $\End_K G$-submodule of $G(K)$ generated by $P_\ell$. Then, by following Steps 2, 3 and 4 of the proof of Theorem~\ref{minmax} (where $P=P_\ell$), we find a set  $\{R_j\}_{j\in J_\ell}$ of independent points such that, by properly prescribing the values of $\ord_{\ell}(R_j \bmod \mathfrak p)$ for every $j\in J_\ell$, we get:
$$\min_{i=1,\ldots, m} \{\ord_{\ell}(Q_i \bmod \mathfrak p)\} > \ord_{\ell}(P_\ell \bmod \mathfrak p)\,.$$
By applying Lemma~\ref{XXX}, we find a positive density of primes $\mathfrak p$ of $K$ such that the above inequality holds for every $\ell\in S$. This contradicts the condition in the statement.
\hfill $\square$ \newline

\section*{Acknowledgements}
We thank Peter Jossen for useful comments and for the counterexample on the support problem for semi-abelian varieties which we mentioned in the Introduction.

\vspace{0.6cm}

\noindent Postdoctoral Fellow of the Research Foundation - Flanders (FWO)

\noindent \textit{E-mail}: antonellaperucca@gmail.com
\end{document}